\documentclass[10pt]{article}
\usepackage[a4paper,hscale=0.7,vscale=0.75,centering]{geometry}

\usepackage[T1]{fontenc}

\usepackage{mathtools}
\usepackage{amssymb}
\usepackage{amsthm}
\usepackage{mathrsfs}

\newtheorem{prop}{Proposition}[section]
\newtheorem{thm}[prop]{Theorem}
\newtheorem{coroll}[prop]{Corollary}
\newtheorem{lemma}[prop]{Lemma}
\newtheorem{defi}[prop]{Definition}
\theoremstyle{remark}
\newtheorem{rmk}[prop]{Remark}
\theoremstyle{definition}

\newcommand{\enne}{\mathbb{N}}
\newcommand{\erre}{\mathbb{R}}

\newcommand{\dom}{\mathsf{D}}
\newcommand{\cA}{\mathscr{A}}
\newcommand{\E}{\mathop{{}\mathbb{E}}\nolimits}
\newcommand{\cF}{\mathscr{F}}
\newcommand{\cL}{\mathscr{L}}
\renewcommand{\P}{\mathbb{P}}

\newcommand{\ep}[1]{{#1}^\varepsilon}
\renewcommand{\geq}{\geqslant}
\renewcommand{\leq}{\leqslant}

\DeclarePairedDelimiter\abs{\lvert}{\rvert}
\DeclarePairedDelimiter\norm{\lVert}{\rVert}
\DeclarePairedDelimiterX\ip[2]{\langle}{\rangle}{#1,#2}

\numberwithin{equation}{section}

\title{On mild solutions to some dissipative SPDEs \\ on \(L^p\) spaces
  with additive noise}

\author{Carlo Marinelli\thanks{%
    Department of Mathematics, University College London, Gower
    Street, London WC1E 6BT, United Kingdom. URL:
    \texttt{http://goo.gl/4GKJP}}}

\date{\normalsize December 29, 2023}

\begin{document}
\maketitle

\begin{abstract}
  We establish well-posedness in the mild sense for a class of
  stochastic semilinear evolution equations on \(L^p\) spaces on
  bounded domains of \(\erre^n\) with a nonlinear drift term given by
  the superposition operator generated by a monotone function on the
  real line with power-like growth. The noise is of additive type with
  respect to a cylindrical Wiener process, with diffusion coefficient
  not necessarily of \(\gamma\)-Radonifying type.
\end{abstract}

%\subjclass[2000]{60H15; 60G57???}

%\keywords{Stochastic PDE, reaction-diffusion equations, monotone operators.}

\section{Introduction}
Consider the stochastic evolution equation in mild form
\begin{equation}
  \label{eq:0}
  u(t) + \int_0^t S(t-s) f(u(s))\,ds = S(t)u_0 + \int_0^t S(t-s)B(s)\,dW(s)
\end{equation}
on a finite time interval $[0,T]$, where \(S\) is a strongly
continuous contraction semigroup on \(L^q:=L^q(G)\), with $G$ a
bounded domain of $\erre^n$ and $q \geq 1$, $f\colon \erre\to\erre$ is
an increasing (possibly discontinuous) function of power-like growth,
$W$ is a cylindrical Wiener noise on a separable Hilbert space $H$,
and \(B\) is a (random) \(\cL(H,L^q)\)-valued process such that the
stochastic integral on the right-hand side of \eqref{eq:0} is a
well-defined \(L^q\)-valued process.
Our main aim is to find criteria for the existence and uniqueness of
solutions to \eqref{eq:0} in terms of integrability conditions of the
stochastic convolution.

Since it is not assumed that \(q \geq 2\) and that the diffusion
coefficient \(B\) takes values in the class of \(\gamma\)-Radonifying
operators \(\gamma(H,L^q)\), approaching the problem by stochastic
calculus techniques (in particular It\^o's formula) seems
hard. Therefore we use a classical idea of changing variable
(informally called ``subtracting the stochastic convolution'') that
reduces the stochastic evolution equation to a deterministic evolution
equation with random drift term. This technique works seamlessly also
for \(q<2\) and only requires the stochastic convolution to be
sufficiently integrable in time and space, but does not need \(B\) to
be \(\gamma(H,L^q)\)-valued. In fact, solutions are constructed
pathwise, hence all arguments would still work, with minimal changes,
if the stochastic convolution were replaced by any stochastic
processes taking values in an \(L^q\) space, thus allowing to treat
equations driven by L\'evy processes, for instance. For the same
reason, there is no need to assume the existence of any moments of the
stochastic convolution. This is likely harder in the stochastic
calculus approach, as it relies on maximal estimates, such as the
Burkholder-Davis-Gundy inequality, that hold in expectation. On the
other hand, it is not clear whether it is possible to adapt the
deterministic method to equations with multiplicative noise, and,
perhaps more importantly, if \(q \geq 2\) and \(B\) is
\(\gamma(H,L^q)\)-valued, so that both approaches are applicable, the
stochastic calculus approach seem to provide better results
(cf. \cite{cm:SIMA18} and Remark \ref{rmk:confr} below).

The literature on semilinear stochastic PDEs is very rich (see, e.g.,
\cite{DPZ} for basic results and references), but the problem
considered here does not seem to fall into the scope of existing
results, at least not entirely. A basic source of difficulties is that
superposition operators on \(L^q\) spaces are not locally Lipschitz
continuous, apart from trivial cases (e.g. if they are linear).
If \(f\) is continuous and \(S\) is an analytic semigroup, comparable
results have been obtained in \cite{KvN2} by methods of stochastic
evolution equations in UMD Banach spaces (see also \cite{cerrai03} for
earlier related results). Moreover, in the case where \(S\) is the
heat semigroup and the stochastic convolution is continuous in space
and time, existence and uniqueness of \(L^1\)-valued pathwise
solutions is discussed in \cite{Barbu:lincei}, under still weaker
assumptions on \(f\), although solutions are only adapted but not
necessarily measurable processes. If \(B\) takes values in
\(\gamma(H,L^2)\), i.e. it is of Hilbert-Schmidt class, well-posedness
results are obtained in \cite{cm:AP18} by a variational approach,
under growth conditions on \(f\) analogous to those of
\cite{Barbu:lincei}. Basic results for the case where \(q=2\), \(f\)
is a polynomial of odd order, and the stochastic convolution is
continuous in space and time and has finite moments of all order are
discussed in \cite[\S4.2]{DP:K}, which has been our main motivation.

The rest of the text is organized as follows: \S\ref{sec:prel} is
dedicated to auxiliary material, most notably estimates for mild
solutions to linear deterministic evolution equations, while
statements and proofs of the main results are the subject of
\S\ref{sec:wp}.

\section{Preliminaries}
\label{sec:prel}
\subsection{Notation}
The sets of positive and strictly positive real numbers will be
denoted by \(\erre_+\) and \(\erre_+^\times\), respectively.
Let \(I\) be an open interval of the real line. The right
Dini derivatives of a function \(f\colon I \to \erre\) at
\(x_0 \in I\) are defined by
\[
  D^+f(x_0) := \limsup_{x \to x_0+} \frac{f(x)-f(x_0)}{x-x_0},
  \qquad
  D_+f(x_0) := \liminf_{x \to x_0+} \frac{f(x)-f(x_0)}{x-x_0},
\]
The left Dini derivatives \(D^-f(x_0)\) and \(D_-f(x_0)\) are defined
analogously, replacing the limits from the right with limits from the
left.

All random quantities are defined on a probability space
\((\Omega,\cF,\P)\) endowed with a complete right-continuous
filtration \((\cF_t)_{t\in[0,T]}\), with \(T \in \erre_+^\times\) a
fixed time horizon. In particular, a fixed cylindrical Wiener process
on a real separable Hilbert space \(H\) will be denoted by \(W\).

Let \(E\) be a Banach space. Its dual will be denoted by \(E'\) and,
if endowed with the \(\sigma(E',E)\) topology, by \(E'_s\). Let \(F\)
be a further Banach space. The vector space of continuous linear maps
from \(E\) to \(F\) is denoted by \(\cL(E,F)\).  If
\(A \subset E \times F\), in particular if \(A\) is a (multivalued)
unbounded operator, the domain of \(A\) is defined by
\(\dom(A) := \{x \in E:\, A(x) \neq \varnothing\}\).
Let \(S\) be a strongly continuous semigroup of linear operators on
\(E\). The (deterministic) convolution of \(S\) with a measurable
function \(\phi\colon [0,T] \to E\) such that
\(S(t-\cdot)\phi \in L^1(0,t;E)\) for every \(t \in [0,T]\) is
defined by
\[
  S \ast \phi := \Bigl( t \mapsto \int_0^t S(t-s)\phi(s)\,ds \Bigr).
\]
Similarly, if \(E\) has the UMD property, the stochastic convolution
of \(S\) with a stochastic process
\(\Phi\colon \Omega \times [0,T] \to \cL(H,E)\) such that
\(S(t-\cdot)\Phi\) is stochastically integrable with respect to \(W\)
is the stochastic process defined by
\[
  S \diamond \Phi := \Bigl( (\omega,t) \mapsto \int_0^t
  S(t-s)\Phi(s)\,dW(s) \Bigr).
\]
We refer to \cite{vNVW} for details on stochastic convolutions and the
semigroup approach to stochastic evolution equations on UMD spaces.

\subsection{Some elementary inequalities}
The proof of the simple inequalities of the following two lemmas are
included for completeness.
\begin{lemma}
  \label{lm:xya}
  Let \(x,y \in \erre_+\). If \(a \in [0,1]\), then
  \[
    2^{a-1} (x^a + y^a)  \leq (x+y)^a \leq x^a + y^a.
  \]
  If \(a \in [1,\infty\mathclose[\), then
  \[
    x^a + y^a \leq (x+y)^a \leq 2^{a-1} (x^a + y^a).
  \]
\end{lemma}
\begin{proof}
  If \(a \in [0,1]\), the function \(x \mapsto x^a\) is concave on
  \(\erre_+\), i.e.
  \[
    \frac{1}{2^a} (x+y)^a = \Bigl( \frac12 x + \frac12 y \Bigr)^a
    \geq \frac12 (x^a + y^a),
  \]
  which proves the lower bound. To prove the upper bound, let us
  assume that \(x+y=1\). This comes at no loss of generality by
  homogeneity. Then it is enough to note that \(x \leq x^a\) and
  \(y \leq y^a\).
  If \(a \in [1,+\infty\mathclose[\), the function \(x \mapsto x^a\)
  is convex on \(\erre_+\), from which the upper bound
  follows. Moreover, assuming that \(x+y=1\) by homogeneity, one has
  \(x \geq x^a\) and \(y \geq y^a\).
\end{proof}

For \(q \in [1,\infty\mathclose[\), let the function
\(j_q\colon \erre \to \erre\) be defined by
\begin{equation}
  \label{eq:jq}
  j_q\colon x \mapsto \abs{x}^{q-1}\operatorname{sgn}(x)
  = \abs{x}^{q-2} x.
\end{equation}
\begin{lemma}
  \label{lm:j}
  For any \(x,y \in \erre\) and \(q \in [1,2]\) one
  has
  \[
    \abs[\big]{j_q(x) - j_q(y)} \lesssim_q \abs{x-y}^{q-1}.
  \]
\end{lemma}
\begin{proof}
  Recall that the function \(x \mapsto x^a\), with \(a \in [0,1]\), is
  \(a\)-H\"older continuous on \(\erre_+\) with constant equal to one.
  If \(x,y \in \erre\) have the same sign, then
  \[
    \abs[\big]{j_q(x) - j_q(y)} = \abs[\big]{\abs{x}^{q-1} -
      \abs{y}^{q-1}} \leq \abs[\big]{\abs{x} - \abs{y}}^{q-1}
    \leq \abs[\big]{x-y}^{q-1}.
  \]
  To consider the case where \(x\) and \(y\) have opposite signs, let
  us assume, without loss of generality, that \(x<0<y\). Then, by
  Lemma \ref{lm:xya},
  \[
    \abs[\big]{j_q(x) - j_q(y)} = y^{q-1} + \abs{x}^{q-1}
    \lesssim_q \bigl( y + \abs{x} \bigr)^{q-1} = \abs{x-y}^{q-1}.
    \qedhere
  \]
\end{proof}

\subsection{Duality mappings}
Throughout this subsection \(E\) is a Banach space.
The duality mapping of \(E\) is the set \(J \subset E \times E'\)
defined by
\[
  J(x) := \bigl\{y \in E':\, \ip{y}{x} = \norm{x}^2 = \norm{y}^2
  \bigr\}.
\]
The duality mapping \(J\) is the subdifferential (in the sense of
convex analysis) of the convex lower semicontinuous function
\(x \mapsto \norm{x}^2/2\), hence \(J\) is a maximal monotone subset
of \(E \times E'\). Moreover, if \(E'\) is strictly convex, then \(J\)
is the graph of a function, i.e. \(J(x)\) is a singleton for every
\(x \in E\), and \(J\colon E \to E'_s\) is continuous.

Let us also introduce, for any \(q \in [1,\infty\mathclose[\), the
set \(J_q \subset E \times E'\) defined by
\[
  J_q(x) := \bigl\{y \in E':\, \ip{y}{x} = \norm{x}^q = \norm{y}^q
  \bigr\}.
\]
Then, as is easily seen, \(J_q(x) = \norm{x}^{q-2} J(x)\) for every
\(x \neq 0\), and \(J_q(0)=0\).
By a theorem of Asplund (see \cite{Asp:dm}), one has
\[
\partial \norm{\cdot}^q = q J_q \qquad \text{on } E \setminus \{0\},
\]
where \(\partial\) stands for the subdifferential.
If \(E=L^q\), with \(q \in \mathopen]1,\infty\mathclose[\), then
\(J_q\colon \phi \mapsto \abs{\phi}^{q-1} \operatorname{sgn}(\phi)\),
that is \(J_q\) is the superposition operator associated to the
function \(j_q\) defined in \eqref{eq:jq}, i.e.
\(J_q\colon \phi \mapsto j_q \circ \phi\).

\subsection{Estimates for mild solutions}
Let us recall a fundamental estimate of the Crandall-Liggett theory of
mild solutions to equations of the type
\[
u' + Cu = F, \qquad u(0)=u_0,
\]
where \(C\) is a (possibly nonlinear and multivalued) \(m\)-accretive
operator on a Banach space \(E\) and \(F \in L^1(0,T;E)\).
To this purpose, let us introduce the so-called bracket
\([\cdot,\cdot]\colon E \times E \to \erre\) defined by
\[
[x,y] := \max_{x^\ast \in J_1(x)} \ip{x^*}{y}.
\]
The following characterization of accretivity in terms of the bracket
is particularly effective.
\begin{lemma}
  A subset \(C\) of \(E \times E\) is accretive if and only if
  \[
    [x-y,Cx-Cy] \geq 0 \qquad \forall x,y \in \dom(C).
  \]
\end{lemma}
\noindent The above-mentioned estimate can now be formulated.
\begin{prop}
  \label{prop:CL}
  Let \(C\) be an \(m\)-accretive subset of \(E \times E\) and
  \(F \in L^1(0,T;E)\). Assume that \(u^i\), \(i=1,2\) are mild
  solutions to
  \[
    (u^i)' + Cu^i = F^i, \qquad u^i(0) = u^i_0, \qquad i=1,2,
  \]
  where \(u^1_0,u^2_0\) belong to the closure of \(\dom(C)\). Then
  \[
    \norm{u^1 - u^2} \leq \norm{u^1_0-u^2_0} + \int_0^\cdot
    \bigl[u^1-u^2,F^1-F^2].
  \]
\end{prop}
\noindent Proofs of these facts and further details can be found in,
e.g., \cite[\S2.3 and pp.~202-ff.]{barbu}

\medskip

In the linear case, without invoking the Crandall-Liggett theory,
similar estimates can be obtained, that in some situations turn out to
be more useful for our purposes. To this aim, we shall need the
following differentiability result, a proof of which is included for
completeness. We shall denote the left and right weak derivatives by
the symbols \(D^-_\sigma\) and \(D^+_\sigma\), respectively.
\begin{prop}
  \label{prop:D}
  Let \(q \in [1,\infty\mathclose[\), \(I \subseteq \erre\) be an open
  interval, \(t \in I\), and \(g\colon I \to E\). If \(g\) is weakly
  right-differentiable at \(t\), then
  \(\norm{g}^q\colon I \to \erre_+\) satisfies
  \[
    D_+ \norm{g(t)}^q \geq q\ip[\big]{y}{D_\sigma^+g(t)} \qquad
    \forall y \in J_q(g(t)).
  \]
  If \(g\) is weakly left-differentiable at \(t\), then
  \[
    D^- \norm{g(t)}^q \leq q\ip[\big]{y}{D_\sigma^-g(t)} \qquad
    \forall y \in J_q(g(t)).
  \]
  In particular, if \(g\) is weakly differentiable at \(t\) and
  \(\norm{g(\cdot)}^q\) is differentiable at \(t\), then
  \[
    D\norm{g(t)}^q =  q\ip[\big]{y}{D_\sigma g(t)} \qquad
    \forall y \in J_q(g(t)).
  \]
\end{prop}
\begin{proof}
  As \(J_q\) is the subdifferential of \(\norm{\cdot}^q/q\), one has,
  for any \(x,\,k \in E\) and any \(z \in J_q(x)\),
  \[
    \norm{x+k}^q - \norm{x}^q \geq q\ip{z}{k}.
  \]
  Let \(h \in \erre_+^\times\) be such that \(t+h \in I\). Taking
  \[
    x := g(t), \qquad k := g(t+h)-g(t), \qquad z := y \in J(g(t))
  \]
  yields
  \[
    \frac{\norm{g(t+h)}^q - \norm{g(t)}^q}{h} \geq
    q\ip[\bigg]{y}{\frac{g(t+h)-g(t)}{h}},
  \]
  hence
  \[
    D_+ \norm{g(t)}^q
    = \liminf_{h \to 0+} \frac{\norm{g(t+h)}^q - \norm{g(t)}^q}{h}
    \geq q\ip[\big]{y}{D^+_\sigma g(t)}.
  \]
  The case of the limit from the left is entirely analogous: one gets
  \[
    D^- \norm{g(t)}^q
    = \limsup_{h \to 0-} \frac{\norm{g(t+h)}^q - \norm{g(t)}^q}{h} \leq
    q\ip[\big]{y}{D^-_\sigma g(t)}.
  \]
  If \(g\) and \(\norm{g(\cdot)}^q\) are weakly differentiable and
  differentiable, respectively, at \(t\), the claim follows by
  comparison.
\end{proof}

Let \(S\) be a strongly continuous contraction semigroup of linear
operators on \(E\) with negative generator \(A\), that, as is well
known, is necessarily \(m\)-accretive.
\begin{prop}
  \label{prop:stm}
  Let \(q \in [1,\infty\mathclose[\), \(E'\) strictly convex,
  \(F \in L^1(0,T;E)\), \(v_0 \in E\), and \(v \in C([0,T];E)\) be
  defined by
  \[
    v(t) = S(t)v_0 + \int_0^t S(t-s)F(s)\,ds.
  \]
  Then \(\zeta_q := J_q(v)\) belongs to \(C([0,T];E'_s)\) and
  \[
    \norm{v}^q \leq \norm{v_0}^q
    + q \int_0^\cdot \ip{F(s)}{\zeta_q(s))}\,ds.
  \]
\end{prop}
\begin{proof}
  Let us first assume that \(v_0 \in \dom(A)\) and
  \(F \in L^1(0,T;\dom(A))\). Then \(v\) is a strong solution,
  i.e. \(v\) is differentiable a.e. on \([0,T]\) and satisfies
  \begin{equation}
    \label{eq:vF}
    v^\prime + Av = F \qquad \text{a.e.},
  \end{equation}
  thus also
  \[
    v = v_0 - \int_0^\cdot Av(s)\,ds + \int_0^\cdot F(s)\,ds,
  \]
  from which it immediately follows that \(v\) is Lipschitz continuous
  (with constant depending on the \(L^1(0,T;\dom(A))\) norms of \(v\)
  and \(F\)). Therefore \(\norm{v(\cdot)}\) is differentiable almost
  everywhere by Rademacher's theorem, and the same holds for
  \(\norm{v(\cdot)}^q\) by the chain rule.
  Since \(E'\) is strictly convex, the duality map \(J_q\) is
  single-valued, hence, setting \(\zeta_q := J_q(v)\) and dualizing
  \eqref{eq:vF} with \(\zeta_q\), one has, by
  Proposition~\ref{prop:D},
  \[
    \frac1q \bigl( \norm{v}^q \bigr)' + \ip{Av}{\zeta_q} = \ip{F}{\zeta_q}
    \qquad \text{a.e.},
  \]
  where, by the accretivity of \(A\), \(\ip{Av}{\zeta_q} \geq
  0\). Moreover, recalling that the strict convexity of \(E'\) also
  implies that \(J_q\colon E \to E'_s\) is continuous, one has, by
  composition, that \(\zeta_q \in C([0,T];E'_s)\).  As the duality
  form \(\ip{\cdot}{\cdot}\colon E \times E'_s \to \erre\) is
  continuous, hence measurable, it follows, by composition, that
  \(\ip{F}{\zeta_q}\colon [0,T] \to \erre\) is measurable. Therefore,
  integrating,
  \[
    \norm{v}^q \leq \norm{v_0}^q + q\int_0^\cdot
    \ip{F(s)}{\zeta_q(s))}\,ds.
  \]
  If \(v_0\) and \(F\) take values in \(E\), but not necessarily in
  \(\dom(A)\), \(v\) does not satisfy \eqref{eq:vF} in general. In
  this case, one can proceed by a regularization step via the
  resolvent of \(A\).  For any \(\varepsilon \in \erre_+^\times\) and
  any \(E\)-valued element \(h\), let us use the notation
  \(h^\varepsilon := (I+\varepsilon A)^{-1}h\).  One has
  \(\ep{v} = S(t)\ep{v}_0 + S \ast \ep{F}\), where
  \(v_0^\varepsilon \in \dom(A)\) and \(\ep{F} \in L^1(0,T;\dom(A))\),
  hence
  \[
    (v^\varepsilon)^\prime + Av^\varepsilon = F^\varepsilon,
    \qquad v^\varepsilon(0) = v_0^\varepsilon,
  \]
  in the strong sense, which implies
  \[
    \norm{v^\varepsilon}^q \leq \norm{v_0^\varepsilon}^q +
    q\int_0^\cdot \ip[\big]{F^\varepsilon(s)}{\zeta^\varepsilon(s))}\,ds,
  \]
  where \(\zeta^\varepsilon = J_q(v^\varepsilon)\).  Let us now pass
  to the limit as \(\varepsilon \to 0\): by well-known properties of
  the resolvent, \(v_0^\varepsilon\) converges strongly to \(v_0\),
  \(F^\varepsilon\) converges strongly to \(F\) a.e. on \([0,T]\), and
  \(v^\varepsilon\) converges strongly to \(v\) pointwise. Therefore
  \(J_q(v^\varepsilon)\) converges to \(J_q(v)\) pointwise in the
  \(\sigma(E',E)\) topology, hence
  \(\ip{F^\varepsilon}{\zeta^\varepsilon}\) converges pointwise to
  \(\ip{F}{\zeta}\). Moreover, by the contractivity of the resolvent,
  \[
    \abs[\big]{\ip[\big]{F^\varepsilon(s)}{\zeta^\varepsilon(s)}}
    \leq \norm[\big]{F^\varepsilon(s)} \,
    \norm[\big]{J(v^\varepsilon(s)}
    \leq \norm[\big]{F(s)} \,
    \norm[\big]{v(s)},
  \]
  where \(s \mapsto \norm{F(s)} \norm{v(s)} \in L^1(0,T)\) because
  \(F \in L^1(0,T;E)\) and \(v \in C([0,T];E)\). The dominated
  convergence theorem then yields
  \[
    \lim_{\varepsilon \to 0}
    \int_0^t \ip{F^\varepsilon(s)}{\zeta^\varepsilon(s)}\,ds
    = \int_0^t \ip{F(s)}{\zeta(s)}\,ds
  \]
  for every \(t \in \erre_+\).
\end{proof}

\subsection{A null sequence}
For the purposes of this subsection, let \((X,\cA,\mu)\) be a finite
measure space, denote \(L^1(X,\cA,\mu)\) simply by \(L^1\), and
analogously for \(L^\infty\). Moreover, let \(\ip{\cdot}{\cdot}\)
stand for the duality between \(L^1\) and \(L^\infty\). The following
lemma is needed in the proof of Proposition \ref{prop:C1}.
\begin{lemma}
  \label{lm:EIconv}
  Let \({(f_n,g_n\colon X \to \erre)}_{n\in\enne}\) be sequences of
  measurable functions. If \((f_n)\) is equiintegrable and \((g_n)\)
  is bounded in \(L^\infty\) and converges to zero in measure, then
  \(\lim_{n\to\infty} \ip{f_n}{g_n} = 0\).
\end{lemma}
\begin{proof}
  Let \(a_n := \ip{f_n}{g_n}\) and
  \(M := \sup_n \norm{f_n}_{L^1} + \sup_n \norm{g_n}_{L^\infty}\). We
  are going to show that every subsequence \((a_{n'})\) admits a
  further subsequence \((a_{n''})\) converging to zero. As is well
  known, this will imply that \((a_n)\) converges to zero. Let then
  \((a_{n'})\) be a subsequence of \((a_n)\). The sequence \(g_{n'}\)
  admits a subsequence \(g_{n''}\) converging to zero almost
  everywhere. Let \(\varepsilon \in \erre_+^\times\). The sequence
  \((f_{n''})\) is equiintegrable, hence there exists
  \(\delta = \delta(\varepsilon) \in \erre_+^\times\) such that
  \[
    \ip{\abs{f}}{1_B} < \frac{\varepsilon}{2M} \qquad \forall f \in (f_{n''})
  \]
  for every \(B \in \cA\) with \(\mu(B)<\delta\).  By the
  Severini-Egorov theorem, \((g_{n''})\) converges to zero almost
  uniformly. That is, there exists a set \(A \in \cA\) with
  \(\mu(A^\complement) < \delta\) such that \((g_{n''})\) converges
  uniformly to zero on \(A\). Then, writing \(k\) in place of \(n''\)
  for simplicity,
  \begin{align*}
    \abs[\big]{\ip{f_k}{g_k}}
    &\leq \abs[\big]{\ip{f_k}{g_k1_A}}
      + \abs[\big]{\ip{f_k}{g_k1_{A^\complement}}}\\
    &\leq \norm[\big]{f_k}_{L^1} \norm[\big]{g_k}_{L^\infty(A)}
      + \norm[\big]{g_k}_{L^\infty} \ip{\abs{f_k}}{1_{A^\complement}}\\
    &\leq M \norm[\big]{g_k}_{L^\infty(A)} + \varepsilon/2.
  \end{align*}
  Since \(g_k\) converges to zero uniformly on \(A\) as
  \(k \to \infty\), there exists \(k_0 \in \enne\) such that
  \(\norm{g_k}_{L^\infty(A)} < \varepsilon/(2M)\) for every
  \(k>k_0\). This proves that \((a_{n''})\) converges to zero, which
  in turn establishes the claim.
\end{proof}

\section{Existence and uniqueness of solutions}
\label{sec:wp}
The following assumptions and conventions are assumed to hold
throughout. Let \(G\) be a bounded domain of \(\erre^n\). For any
\(q \in [1,+\infty\mathclose[\), we shall write \(L^q\) to mean
\(L^q(G)\).
Let \(A\) be a linear (unbounded) \(m\)-accretive
operator on \(L^1\), that can be restricted to an operator of the same
class on every \(L^q\), \(q \in \mathopen]1,+\infty\mathclose[\). We
shall not notationally distinguish between realizations of \(A\) on
different \(L^q\) spaces. The strongly continuous semigroup of
contractions generated by \(-A\) will be denoted by \(S\), again
without explicit indication of the underlying \(L^q\) space (this is
harmless, as the family of semigroups is known to be consistent
because of the assumptions on \(A\)).  This assumption is actually too
strong for our needs: inspecting the proofs one can extrapolate on
which \(L^q\) spaces \(A\) should generate (cf. \cite{cm:SIMA18}), but
we do not do it here for the sake of simplicity.
The function \(f\colon \erre \to \erre\) is increasing and there
exists \(d \in \erre_+\) such that
\[
  \abs{f(x)} \lesssim 1 + \abs{x}^d \qquad \forall x \in \erre.
\]
The diffusion coefficient \(B\) is a strongly measurable adapted
process taking values in \(\cL(H,E)\), where \(E\) is a Banach space
such that \(S \diamond B\) is a well-defined \(L^q\) valued process,
with the value of \(q\) depending on the concept of solution (see
below). For instance, if \(S\) is analytic, \(E\) can be the domain of
a negative power of \(A\) (cf. \cite{vNVW} for several examples). As
the main results are formulated in terms of assumptions on
\(S \diamond B\), characterizing \(E\) is irrelevant.

We can now define solutions to \eqref{eq:0}. To this purpose, we need
to recall that any increasing function \(\phi\colon \erre \to \erre\)
can be uniquely extended to a maximal monotone graph
\(\widetilde{\phi} \subset \erre \times \erre\), by the procedure of
``filling the jumps'': for any \(x \in \erre\) one sets
\(\widetilde{\phi}(x) = [\phi(x-),\phi(x+)]\). We shall not
distinguish notationally between \(f\) and its unique extension to a
maximal monotone graph of \(\erre \times \erre\).

\begin{defi}
  Let \(q,r \in [1,+\infty\mathclose[\) with \(q \geq r\) and
  \(u_0 \in L^0(\cF_0;L^q)\). A \((q,r)\)-mild solution to
  \eqref{eq:0} is an adapted process
  \(u \in L^0(\Omega;C([0,T];L^q))\) such that there exists an adapted
  process \(g \in L^0(\Omega;L^1(0,T;L^r))\), with \( g \in f(u)\)
  a.e. in \(\Omega \times [0,T] \times G\) and
  \[
    u + S \ast g = Su_0 + S \diamond B,
  \]
  in the sense of indistinguishable \(L^r\)-valued processes. A
  \emph{strict mild} solution and a \emph{mild} solution are a
  \((q,q)\)-mild and a \((q,1)\)-mild solution, respectively.
\end{defi}
\noindent We implicitly intend, as part of the definition, that the
stochastic convolution is a well-defined \(L^r\)-valued process.

\begin{defi}
  Let \(q \in [1,+\infty\mathclose[\) and \(u_0 \in L^0(\cF_0;L^q)\).
  A continuous \(L^q\)-valued adapted process \(u\) is a
  \emph{generalized mild} solution to \eqref{eq:0} if it is the limit
  in \(L^0(\Omega;C([0,T];L^q))\) of a sequence of strict mild
  solutions.
\end{defi}

Solutions will be constructed using the classical scheme of
regularizing \(f\) by its Yosida approximation, about which we recall
some basic facts.
The family
\({(f_\lambda\colon \erre \to \erre)}_{\lambda \in \erre_+^\times}\)
of Yosida approximations of \(f\) is defined by
\[
f_\lambda := \frac{1}{\lambda} \bigl( I - (I+\lambda f)^{-1} \bigr).
\]
As is well known, \(f_\lambda\) is monotone and Lipschitz continuous on
\(\erre\), hence also on \(L^q\) when viewed as a superposition
operator, and satisfies \(\abs{f_\lambda(x)} \leq \abs{f(x)}\) for
every \(x \in \erre\).
The family
\({(R_\lambda\colon \erre \to \erre)}_{\lambda \in \erre_+^\times}\)
of resolvents of \(f\) is defined by
\[
R_\lambda = (I+\lambda f)^{-1}.
\]
We shall repeatedly use the identity, valid for all \(x,y \in \erre\)
and \(\lambda,\mu \in \erre_+^\times\),
\begin{equation}
  \label{eq:lm}
  \begin{split}
  x - y
  &= R_\lambda x - R_\mu y
    + x - R_\lambda x - (y - R_\mu y)\\
  &= R_\lambda x - R_\mu y
    + \lambda f_\lambda(x) - \mu f_\mu(y)
  \end{split}
\end{equation}
and the inequality
\begin{equation}
  \label{eq:flm}
  \begin{split}
  \bigl( f_\lambda(x) - f_\mu(y) \bigr)(x-y)
  &\in \bigl( f(R_\lambda x) - f(R_\mu y) \bigr) (R_\lambda x - R_\mu y)\\
  &\quad + \bigl( f_\lambda(x) - f_\mu(y) \bigr)%
  \bigl( \lambda f_\lambda(x) - \mu f_\mu(y) \bigr)\\
  &\geq \bigl( f_\lambda(x) - f_\mu(y) \bigr)%
  \bigl( \lambda f_\lambda(x) - \mu f_\mu(y) \bigr)\\
  &\gtrsim -(\lambda+\mu) \bigl( \abs{f_\lambda(x)}^2 + \abs{f_\mu(y)}^2
  \bigr).
  \end{split}
\end{equation}

Consider the regularized equation
\begin{equation}
  \label{eq:reg}
  du_\lambda(t) + Au_\lambda(t)\,dt + f_\lambda(u_\lambda(t))\,dt
  = B(t)\,dW(t),
  \qquad u_\lambda(0)=u_0.
\end{equation}
For any \(q \in [1,+\infty\mathclose[\), if \(u_0 \in L^0(\cF_0;L^q)\)
and \(S \diamond B\) is a continuous \(L^q\)-valued process, the
Lipschitz continuity of \(f_\lambda\) implies that \eqref{eq:reg}
admits a unique strict mild solution
\(u_\lambda \in L^0(\Omega;C([0,T];L^q))\), i.e.,
\[
u_\lambda + S \ast f_\lambda(u_\lambda) = Su_0 + S \diamond B.
\]

\subsection{Estimates of solutions to the regularized equation}
We shall need the following integral inequality, the proof of which
simply follows by explicitly solving Bernoulli's ODE (cf., e.g.,
\cite[p.~29]{Walter:ODE}).
\begin{lemma}
  \label{lm:Bern}
  Let \(g \in L^1(0,T;\erre_+)\), \( y_0 \in \erre_+^\times\), and
  \(y \in C([0,T])\) be such that
  \[
  y^2 \leq y_0^2 + \int_0^\cdot gy.
  \]
  Then
  \[
  \abs{y} \leq y_0 + 2\int_0^\cdot g.
  \]
\end{lemma}
\begin{prop}
  \label{prop:ape1}
  Let \(q \in \mathopen]1,+\infty\mathclose[\) and \(p \in \erre_+\). If
  \(u_0 \in L^p(\cF_0;L^q)\) and
  \[
    S \diamond B \in L^p(\Omega;C([0,T];L^q)) \cap
    L^{pd}(\Omega;L^d(0,T;L^{qd})),
  \]
  then \((u_\lambda)\) is bounded in
  \(L^p(\Omega;C([0,T];L^q))\). More precisely, there exists a
  constant \(N\), independent of $\lambda$, such that
  \[
  \E\sup_{t \leq T} \norm[\big]{u_\lambda(t)}_{L^{q}}^p \leq N
  \Bigl( 1 + \E\norm[\big]{u_0}_{L^{q}}^p \Bigr).
  \]
\end{prop}
\begin{proof}
  Let us set, for simplicity of notation, \(z:=S \diamond B\). The
  mild form of the regularized equation can equivalently be written as
  \[
    u_\lambda - z + S \ast f_\lambda(u_\lambda - z + z) = Su_0,
  \]
  hence \(v_\lambda:=u_\lambda - z\) is the unique mild solution to
  the deterministic evolution equation with random coefficients
  \begin{equation}
    \label{eq:vlz}
    v'_\lambda + Av_\lambda + f_\lambda(v_\lambda + z) = 0,
    \qquad v_\lambda(0)=u_0.
  \end{equation}
  By Proposition \ref{prop:stm}, setting \(\zeta := J_q(v_\lambda)\)
  and denoting the \(L^q\) norm by \(\norm{\cdot}\), one has
  \[
    \norm{v_\lambda}^2 \leq \norm{u_0}^2
    - 2\int_0^\cdot \ip[\big]{f_\lambda(v_\lambda + z)}{\zeta)},
  \]
  where, by monotonicity of \(f_\lambda\),
  \(\ip[\big]{f_\lambda(v_\lambda+z) - f_\lambda(z)}{\zeta} \geq 0\), hence
  \[
    \ip[\big]{f_\lambda(v_\lambda + z)}{\zeta)}
    \geq \ip[\big]{f_\lambda(z)}{\zeta}
    \geq - \norm[\big]{f_\lambda(z)}_{L^q} \norm[\big]{\zeta}_{L^{q'}}.
  \]
  Therefore, recalling that
  \({\norm{\zeta}}_{L^{q'}} = {\norm{v_\lambda}}_{L^q}\),
  \[
  \norm{v_\lambda}^2 \leq \norm{u_0}^2 
  + 2 \int_0^\cdot \norm{f_\lambda(z} \norm{v_\lambda},
  \]
  which implies, by Lemma \ref{lm:Bern} and the inequality
  \(\abs{f_\lambda} \leq \abs{f}\),
  \[
    \norm{v_\lambda} \leq \norm{u_0} + 4 \int_0^\cdot
    \norm{f_\lambda(z)} \leq \norm{u_0} + 4 \int_0^\cdot \norm{f(z)},
  \]
  thus also, in view of \(\abs{f} \lesssim 1 + \abs{\cdot}^d\),
  \begin{align*}
    \norm{u_\lambda}
    &\leq \norm{u_0} + \norm{z} + 4 \int_0^\cdot \norm{f(z)}\\
    &\lesssim \norm{u_0} + \norm{z} + \int_0^\cdot \bigl(
      \abs{G}^{1/q} + \norm[\big]{z}^d_{L^{qd}} \bigr).
  \end{align*}
  It immediately follows that
  \[
    \norm[\big]{u_\lambda}_{C([0,T];L^q)} \lesssim T \abs{G}^{1/q} +
    \norm{u_0} + \norm[\big]{z}_{C([0,T];L^q)} +
    \norm[\big]{z}^d_{L^d(0,T;L^{qd})}
  \]
  as well as, for every \(p \in \erre_+\),
  \[
    \norm[\big]{u_\lambda}_{L^p(\Omega;C([0,T];L^q))}
    \lesssim T\abs{G}^{1/q} + \norm[\big]{u_0}_{L^p(\Omega;L^q)}
    + \norm[\big]{z}_{L^p(\Omega;C([0,T];L^q))}
    + \norm[\big]{z}^d_{L^{pd}(\Omega;L^d(0,T;L^{qd})}.
    \qedhere
  \]
\end{proof}

\begin{rmk}
  It would be more satisfying, at least aesthetically, to have an
  argument allowing to prove the statement of Proposition
  \ref{prop:ape1} also for \(q=1\). In this case Proposition
  \ref{prop:stm} is no longer applicable because \(L^1\) does not have a
  strictly convex dual. Trying instead to apply Proposition
  \ref{prop:CL}, one arrives at
  \[
    \norm{v_\lambda} \leq \norm{u_0} + \int_0^\cdot \bigl[
    v_\lambda,-f_\lambda(v_\lambda+z) \bigr],
  \]
  from where it is unclear how to proceed, as the properties of the
  bracket are too weak to produce usable estimates. On the other hand,
  what is typically needed are estimates for \(q\) sufficiently large,
  so this is not a serious limitation.
\end{rmk}

If \(q \geq 2\), another estimate can be obtained, which requires less
integrability in space and with respect to \(\P\), but slightly more
integrability in time.
\begin{prop}
  \label{prop:ape2}
  Let \(q \in [2,+\infty\mathclose[\) and \(p \in \erre_+\). If
  \(u_0 \in L^p(\cF_0;L^q)\) and
  \[
    S \diamond B \in L^p(\Omega;C([0,T];L^q)) \cap
    L^{p(d+1)/2}(\Omega;L^{d+1}(0,T;L^{q(d+1)/2}))),
  \]
  then \((u_\lambda)\) is bounded in
  \(L^p(\Omega;C([0,T];L^q))\). More precisely, there exists a
  constant $N$, independent of $\lambda$, such that
  \[
  \E\sup_{t \leq T} \norm[\big]{u_\lambda(t)}^p_{L^q} \leq N
  \Bigl( 1 + \E\norm[\big]{u_0}_{L^q}^p \Bigr).
  \]
\end{prop}
\begin{proof}
  Using the same notation of the proof of Proposition \ref{prop:ape1},
  and repeating its first steps, one has
  \[
    \norm{v_\lambda}^2 \leq \norm{u_0}^2
    - 2\int_0^\cdot \ip[\big]{f_\lambda(v_\lambda + z)}{J(v_\lambda)},
  \]
  where
  \[
    \ip[\big]{f_\lambda(v_\lambda + z)}{J(v_\lambda)} =
    \ip[\big]{f_\lambda(v_\lambda + z)}{v_\lambda\abs{v_\lambda}^{q-2}}
    \norm{v_\lambda}^{2-q}_{L^q}.
  \]
  Let \(\varphi\colon \erre \to \erre_+\) be a convex function such
  that \(f=\partial \varphi\). Denoting the Moreau regularization of
  \(\varphi\) by \(\varphi_\lambda\), one has
  \(f_\lambda = \varphi'_\lambda = \partial\varphi_\lambda\), hence,
  by the convexity of \(\varphi_\lambda\) and the definition of
  subdifferential,
  \[
  f_\lambda(v_\lambda + z)v_\lambda \geq \varphi_\lambda(v_\lambda + z)
  - \varphi_\lambda(z) \geq -\varphi_\lambda(z),
  \]
  therefore
  \[
    \norm{v_\lambda}^2 \leq \norm{u_0}^2 - 2\int_0^\cdot
    \ip[\big]{\varphi_\lambda(z)}{\abs{v_\lambda}^{q-2}}
    \norm{v_\lambda}^{2-q}
  \]
  H\"older's inequality with conjugate exponents \(q/2\) and
  \(q/(q-2)\) yields
  \begin{align*}
  \ip[\big]{\varphi_\lambda(z)}{\abs{v_\lambda}^{q-2}} 
  &\leq \norm[\big]{\varphi_\lambda(z)}_{L^{q/2}}
  \norm[\big]{\abs{v_\lambda}^{q-2}}_{L^{\frac{q}{q-2}}}\\
  &\leq \norm[\big]{\varphi(z)}_{L^{q/2}}
  \norm[\big]{v_\lambda}_{L^q}^{q-2},%\\
  % &= \norm[\big]{\varphi^{1/2}(S \diamond B)}_{L^q}^2
  % \norm[\big]{v_\lambda}_{L^q}^{q-2}
  \end{align*}
  hence
  \[
    \norm{v_\lambda}^2 \leq \norm{u_0}^2 + 2 \int_0^\cdot
    \norm[\big]{\varphi(z)}_{L^{q/2}},
  \]
  which in turn implies
  \[
    \norm[\big]{v_\lambda}_{C([0,T];L^q)} \leq \norm{u_0} + \sqrt{2}
    \norm[\big]{\varphi(z)}^{1/2}_{L^1(0,T;L^{q/2})}.
  \]
  Recalling that $\partial\varphi = f$ and
  $\abs{f} \lesssim 1 + \abs{\cdot}^d$, the mean value theorem for
  convex functions %of \S\ref{ssec:conv}
  implies $\abs{\varphi} \lesssim 1 + \abs{\cdot}^{d+1}$, hence
  \[
    \norm[\big]{v_\lambda}_{C([0,T];L^q)} \lesssim
    T^{1/2} \abs{G}^{1/4} + \norm{u_0} +
    \norm[\big]{z}^{(d+1)/2}_{L^{d+1}(0,T;L^{(d+1)q/2})},
  \]
  thus also
  \begin{align*}
    \norm[\big]{u_\lambda}_{L^p(\Omega;C([0,T];L^q))}
    &\lesssim T^{1/2} \abs{G}^{1/4} + \norm[\big]{u_0}_{L^p(\Omega;L^q)}\\
    &\quad + \norm[\big]{z}_{L^p(\Omega;C([0,T];L^q))}
      + \norm[\big]{z}^{(d+1)/2}_{L^{p(d+1)/2}(\Omega;L^{d+1}(0,T;L^{(d+1)q/2}))}.
      \qedhere
  \end{align*}
\end{proof}

\begin{rmk}
  \label{rmk:confr}
  If \(q \geq 2\) and \(B\) is \(\gamma\)-Radonifying, better
  estimates on \((u_\lambda)\) can be obtained by stochastic calculus
  techniques. In fact, it is shown in \cite{cm:SIMA18} that, in this
  case,
  \[
    \norm[\big]{(u_\lambda)}_{L^p(\Omega;C([0,T];L^q))} \lesssim 1 +
    \norm[\big]{u_0}_{L^p(\Omega;L^q)} +
    \norm[\big]{B}_{L^p(L^2(0,T;\gamma(H;L^q)))}.
  \]
  The same estimate clearly holds with \((u_\lambda)\) replaced by
  \(S \diamond B\), hence also with \((u_\lambda)\) replaced by
  \((v_\lambda)\). Unfortunately, however, we have not been able to
  obtain such an estimate starting from the equation \eqref{eq:vlz}
  satisfied by \(v_\lambda\) and using the deterministic techniques
  employed so far.
\end{rmk}

\subsection{Convergence of approximating solutions}
We are going to determine sufficient conditions for solutions
\((u_\lambda)\) to the regularized equation \eqref{eq:reg} to form a
Cauchy sequence in spaces of continuous \(L^q\)-valued processes,
treating the cases \(q \in [2,+\infty\mathclose[\),
\(q \in \mathopen]1,2\mathclose[\), and \(q=1\) separately. To this
purpose, note that, for any \(\lambda,\mu \in \erre_+^\times\),
\(u_\lambda-u_\mu\) satisfies the identity
\[
  (u_\lambda-u_\mu)' + A(u_\lambda - u_\mu)
  + f_\lambda(u_\lambda) - f_\mu(u_\mu) = 0,
  \qquad (u_\lambda-u_\mu)(0)=0
\]
in the mild sense.
\begin{lemma}
  \label{lm:C2+}
  Let \(q \in [2,\infty\mathclose[\), \(p \in \erre_+\), and
  \[
    p^\ast := \frac{p(2d+q-2)}{q}.
  \]
  If \((u_\lambda)\) is bounded in
  \(L^{p^\ast}(\Omega;L^{2d+q-2}([0,T] \times G))\), then
  it is a Cauchy sequence in
  \(L^p(\Omega;C([0,T];L^q))\).
\end{lemma}
\begin{proof}
  Let \(\lambda,\mu \in \erre_+^\times\). Denoting the \(L^q\) norm by
  \(\norm{\cdot}\), Proposition \ref{prop:stm} yields
  \[
    \norm[\big]{u_\lambda - u_\mu}^q \leq -q \int_0^\cdot
    \ip[\big]{f_\lambda(u_\lambda)-f_\mu(u_\mu)}{J_q(u_\lambda-u_\mu)},
  \]
  where
  \[
    \ip[\big]{f_\lambda(u_\lambda)-f_\mu(u_\mu)}{J_q(u_\lambda-u_\mu)}
    = \ip[\big]{f_\lambda(u_\lambda)-f_\mu(u_\mu)}%
      {(u_\lambda - u_\mu)\abs{u_\lambda - u_\mu}^{q-2}}
  \]
  and, by \eqref{eq:flm},
  \begin{align*}
  &\ip[\big]{f_\lambda(u_\lambda)-f_\mu(u_\mu)}%
    {(u_\lambda - u_\mu)\abs{u_\lambda - u_\mu}^{q-2}}\\
  &\hspace{3em} \geq \ip[\big]{f_\lambda(u_\lambda)-f_\mu(u_\mu)}%
    {(\lambda f_\lambda(u_\lambda) - \mu f_\mu(u_\mu))%
    \abs{u_\lambda - u_\mu}^{q-2}}\\
  &\hspace{3em} \gtrsim - (\lambda+\mu) \ip[\big]{\abs{f_\lambda(u_\lambda)}^2%
    + \abs{f_\mu(u_\mu)}^2}{\abs{u_\lambda - u_\mu}^{q-2}},
  \end{align*}
  hence
  \[
    \norm{u_\lambda - u_\mu}^q \lesssim q (\lambda + \mu) \int_0^\cdot
    \ip[\big]{\abs{f_\lambda(u_\lambda)}^2 + \abs{f_\mu(u_\mu)}^2}%
    {\abs{u_\lambda - u_\mu}^{q-2}}.
  \]
  The growth condition on \(f\) and elementary estimates imply
  \[
    \bigl( \abs{f_\lambda(x)}^2 + \abs{f_\mu(y)}^2 \bigr) \abs{x - y}^{q-2}
    \lesssim 1 + \abs{x}^{2d+q-2} + \abs{y}^{2d+q-2}
  \]
  with implicit constant depending on \(d\) and \(q\). This in turn
  implies
  \[
  \ip[\big]{\abs{f_\lambda(u_\lambda)}^2%
    + \abs{f_\mu(u_\mu)}^2}{\abs{u_\lambda - u_\mu}^{q-2}}
  \lesssim 1 + \norm[\big]{u_\lambda}^{2d+q-2}_{L^{2d+q-2}}
  + \norm[\big]{u_\mu}^{2d+q-2}_{L^{2d+q-2}}
  \]
  with implicit constant depending on \(d\), \(q\), and (the Lebesgue
  measure of) \(G\), hence
  \[
    \norm[\big]{u_\lambda-u_\mu}^q_{C([0,T];L^q)} \lesssim (\lambda+\mu)
    \Bigl( T + \norm[\big]{u_\lambda}^{2d+q-2}_{L^{2d+q-2}([0,T] \times G)}
      + \norm[\big]{u_\mu}^{2d+q-2}_{L^{2d+q-2}([0,T] \times G)} \Bigr),
  \]
  which yields
  \[
    \norm[\big]{u_\lambda-u_\mu}_{C([0,T];L^q)} \lesssim
    (\lambda+\mu)^{1/q} \Bigl( T^{1/q} +
    \norm[\big]{u_\lambda}^{\frac{2d+q-2}{q}}_{L^{2d+q-2}([0,T] \times G)} +
    \norm[\big]{u_\mu}^{\frac{2d+q-2}{q}}_{L^{2d+q-2}([0,T] \times G)} \Bigr).
  \]
  Noting that
  \[
    \norm[\Big]{\norm[\big]{u_\lambda}^{\frac{2d+q-2}{q}}_{L^{2d+q-2}([0,T]
        \times G)}}_{L^p(\Omega)} =
    \norm[\big]{u_\lambda}^{\frac{2d+q-2}{q}}_{L^{p^\ast}(\Omega;L^{2d+q-2}([0,T]
      \times G))}
  \]
  completes the proof.
\end{proof}

\begin{rmk}
  The same estimate could have been obtained invoking Proposition
  \ref{prop:stm} with exponent two, thus using the ``standard''
  duality map. In this case, however, the term
  \(\norm{u_\lambda-u_\mu}^{2-q}\) would appear, making computations
  somewhat more cumbersome. In particular, in order to reach the
  desired conclusion, a differential inequality of the type
  \(y' \leq g(s)y^{-\alpha}\), with \(\alpha\) and \(g\) a positive
  constant and a positive function, respectively, needs to be
  solved.
\end{rmk}

\begin{rmk}
  A less sharp sufficient condition for the claim of Lemma
  \ref{lm:C2+} to hold can be obtained by H\"older's inequality with
  conjugate exponents \(q/2\) and \(q/(q-2)\), that yields the estimate
  \[
    \ip[\big]{\abs{f_\lambda(u_\lambda)}^2 + \abs{f_\mu(u_\mu)}^2}%
             {\abs{u_\lambda - u_\mu}^{q-2}}
    \leq \bigl( \norm{f_\lambda(u_\lambda)}^q + \norm{f_\mu(u_\mu)}^q
         \bigr)
    \norm{u_\lambda - u_\mu}^{q-2},
  \]
  which implies
  \[
    \norm{u_\lambda-u_\mu}^2 \lesssim (\lambda+\mu) \int_0^\cdot \bigl(
    \norm{f_\lambda(u_\lambda)}^q + \norm{f_\mu(u_\mu)}^q \bigr).
  \]
  One would then need \((u_\lambda)\) to be bounded in
  \(L^{qd}([0,T] \times G)\) in order for it to be a Cauchy sequence
  in \(C([0,T];L^q)\). Since \(q \geq 2\) implies that
  \(2d+q-2 \leq qd\) for every \(d \geq 0\), boundedness in
  \(L^{qd}([0,T] \times G)\) is a stronger requirement than that of
  the previous lemma.
\end{rmk}

Let us now consider the case \(q \in \mathopen]1,2\mathclose[\).
\begin{lemma}
  \label{lm:C1+}
  Let \(q \in \mathopen]1,2\mathclose[\) and \(p \in \erre_+\). If
  \((u_\lambda)\) is bounded in
  \(L^{pd}(\Omega;L^{qd}([0,T] \times G))\), then it is a Cauchy
  sequence in \(L^p(\Omega;C([0,T];L^q))\).
\end{lemma}
\begin{proof}
  Proposition \ref{prop:stm} yields, for any
  \(\lambda,\mu \in \erre_+^\times\),
  \[
    \norm{u_\lambda-u_\mu}^q + \int_0^\cdot
    \ip[\big]{f_\lambda(u_\lambda)-f_\mu(u_\mu)}{J_q(u_\lambda-u_\mu)} \leq 0,
  \]
  where, recalling \eqref{eq:lm},
  \begin{align*}
  &\ip[\big]{f_\lambda(u_\lambda)-f_\mu(u_\mu)}{J_q(u_\lambda-u_\mu)}\\
  &\hspace{3em} = \ip[\big]{f(R_\lambda u_\lambda) - f(R_\mu u_\mu)}%
    {J_q(R_\lambda u_\lambda - R_\mu u_\mu)}\\
  &\hspace{3em}\quad + \ip[\big]{f_\lambda(u_\lambda)-f_\mu(u_\mu)}%
    {J_q(u_\lambda-u_\mu) - J_q(R_\lambda u_\lambda - R_\mu u_\mu)}\\
  &\hspace{3em}\geq \ip[\big]{f_\lambda(u_\lambda)-f_\mu(u_\mu)}%
    {J_q(u_\lambda-u_\mu) - J(R_\lambda u_\lambda - R_\mu u_\mu)}.
  \end{align*}
  Since, by Lemma \ref{lm:j},
  \(\abs[\big]{J_q(\phi)-J_q(\psi)} \lesssim_q
  \abs[\big]{\phi-\psi}^{q-1}\) for any \(\phi,\psi \in L^q\), one has
  \begin{align*}
    \abs[\big]{J_q(u_\lambda-u_\mu) - J_q(R_\lambda u_\lambda - R_\mu u_\mu)}
    &\lesssim_q \abs[\big]{%
      u_\lambda - R_\lambda u_\lambda - (u_\mu - R_\mu u_\mu)}^{q-1}\\
    &= \abs[\big]{\lambda f_\lambda(u_\lambda) - \mu f_\mu(u_\mu)}^{q-1}\\
    &\leq (\lambda+\mu)^{q-1} \bigl( \abs{f_\lambda(u_\lambda)}
      + \abs{f_\mu(u_\mu)} \bigr)^{q-1},
  \end{align*}
  which in turn implies
  \begin{align*}
    &\ip[\big]{f_\lambda(u_\lambda)-f_\mu(u_\mu)}{J_q(u_\lambda-u_\mu)}\\
    &\hspace{3em} \gtrsim_q - (\lambda + \mu)^{q-1} \Bigl(
      \norm[\big]{f_\lambda(u_\lambda)}^{q}
      + \norm[\big]{f_\mu(u_\mu)}^{q} \Bigr)\\
    &\hspace{3em} \gtrsim_{\abs{G}} - (\lambda + \mu)^{q-1} \Bigl(
      1 + \norm[\big]{u_\lambda}_{L^{qd}}^{qd}
      + \norm[\big]{u_\mu}_{L^{qd}}^{qd} \Bigr),
  \end{align*}
  thus also
  \[
    \norm{u_\lambda-u_\mu}^q \lesssim (\lambda + \mu)^{q-1}
    \int_0^\cdot \Bigl( 1 + \norm[\big]{u_\lambda}_{L^{qd}}^{qd} +
    \norm[\big]{u_\mu}_{L^{qd}}^{qd} \Bigr)
  \]
  as well as
  \[
    \norm[\big]{u_\lambda-u_\mu}_{C([0,T];L^q)}^q \lesssim
    (\lambda + \mu)^{q-1} \Bigl( T +
    \norm[\big]{u_\lambda}_{L^{qd}([0,T] \times G)}^{qd} +
        \norm[\big]{u_\mu}_{L^{qd}([0,T] \times G)}^{qd} \Bigr),
  \]
  with an implicit constant depending on \(q\) and \(\abs{G}\). Then
  \[
    \norm[\big]{u_\lambda-u_\mu}_{C([0,T];L^q)} \lesssim (\lambda +
    \mu)^{\frac{q-1}{q}} \Bigl( T^{1/q} +
    \norm[\big]{u_\lambda}_{L^{qd}([0,T] \times G)}^{d} +
    \norm[\big]{u_\mu}_{L^{qd}([0,T] \times G)}^{d} \Bigr),
  \]
  from which the claim follows by virtue of the identity
  \[
    \norm[\Big]{\norm[\big]{u_\lambda}%
      _{L^{qd}([0,T] \times G)}^{d}}_{L^p(\Omega)} =
    \norm[\big]{u_\lambda}^d_{L^{pd}(\Omega;L^{qd}([0,T] \times G))}.
    \qedhere
  \]
\end{proof}

The case \(q=1\) is more involved. We need some approximations of the
sign and the modulus functions that we introduce next.
Let us define the family
\({(\gamma_\varepsilon\colon \erre \to \erre_+)}_{\varepsilon \in
  \erre_+^\times}\) of piecewise linear approximations of the sign
function by
\[
  \gamma_\varepsilon(x) :=
  \begin{cases}
    -1, & x \in \mathopen]-\infty,-\sqrt{\varepsilon}/2\mathclose[,\\
    \displaystyle \frac{1}{\sqrt{\varepsilon}} x,
    & x \in [-\sqrt{\varepsilon}/2,\sqrt{\varepsilon}/2\mathclose],\\
    1, & x \in \mathopen]\sqrt{\varepsilon}/2,+\infty\mathclose[.
  \end{cases}
\]
Then the family of functions
\({(\Gamma^0_\varepsilon\colon \erre \to \erre_+)}_{\varepsilon \in
  \erre_+^\times}\) defined by
\[
  \Gamma_\varepsilon^0(x) = \frac{\sqrt{\varepsilon}}{4}
  + \int_0^x \gamma_\varepsilon(y)\,dy
\]
is a family of convex, even, \(C^1\) approximation of \(\abs{\cdot}\)
satisfying the following properties:
\begin{itemize}
\item[(i)] \(\gamma_\varepsilon := (\Gamma^0_\varepsilon)'\), hence
  \(\Gamma^0_\varepsilon\) is Lipschitz continuous with Lipschitz
  constant equal to one;
\item[(ii)] \(\Gamma^0_\varepsilon(x) = \abs{x}\) for every
  \(x \in \erre \setminus
  \mathopen]-\sqrt{\varepsilon}/2,\sqrt{\varepsilon}/2\mathclose[\);
\item[(iii)] \(\Gamma^0_\varepsilon(x) \geq \abs{x}\) for every
  \(x \in \erre\);
\item[(iv)]
  \(\sup_{x \in \erre} \abs[\big]{\Gamma^0_\varepsilon(x) - \abs{x}} =
  \Gamma^0_\varepsilon(0) = \sqrt{\varepsilon}/4\).
\end{itemize}
Since \(\Gamma^0_\varepsilon\) is Lipschitz continuous for every
\(\varepsilon \in \erre_+^\times\), setting 
\[
  \Gamma_\varepsilon\colon \phi \longmapsto \int_G
  \Gamma^0_\varepsilon \circ \phi.
\]
defines a family of maps
\({(\Gamma_\varepsilon\colon L^1 \to \erre_+)}_{\varepsilon \in
  \erre_+^\times}\). Moreover, as \(\Gamma^0_\varepsilon\) is also
continuously differentiable, it is not hard to see that if
\(\phi\colon [0,T] \to L^1\) is a strongly differentiable map, then
\[
  \bigl( \Gamma_\varepsilon(\phi(t)) \bigr)'
  = \int_G \gamma_\varepsilon(\phi(t)) \phi'(t)
  = \ip[\big]{\gamma_\varepsilon(\phi(t))}{\phi'(t)}
  \qquad \forall t \in [0,T].
\]
\begin{prop}
  \label{prop:C1}
  Assume that \(S\) is subMarkovian. If \((f_\lambda(u_\lambda))\) is
  equiintegrable on \([0,T] \times G\) a.s., then \((u_\lambda)\) is a
  Cauchy sequence in \(C([0,T;L^1)\) a.s.
\end{prop}
\begin{proof}
  Let us set, for notational conciseness,
  \[
    y_{\lambda\mu} := u_\lambda - u_\mu, \qquad
    g_{\lambda\mu} := f_\lambda(u_\lambda) - f_\mu(u_\mu)
  \]
  for any \(\lambda,\mu \in \erre_+^\times\), so that
  \(y'_{\lambda\mu} + Ay_{\lambda\mu} + g_{\lambda\mu}=0\) in the mild
  sense, with \(y_{\lambda\mu}(0)=0\). Setting, for any
  \(\varepsilon \in \erre_+^\times\),
  \[
    y_{\lambda\mu}^\varepsilon := (I+\varepsilon A)^{-1} y_{\lambda\mu},
    \qquad
    g_{\lambda\mu}^\varepsilon := (I+\varepsilon A)^{-1} g_{\lambda\mu},
  \]
  one has
  \((y^\varepsilon_{\lambda\mu})' + Ay^\varepsilon_{\lambda\mu} +
  g^\varepsilon_{\lambda\mu}=0\) in the strong sense. Multiplying
  pointwise both sides of this identity by
  \(\gamma_{\lambda+\mu}(y^\varepsilon_{\lambda\mu})\) and integrating
  over \(G\) yields
  \[
  \frac{d}{dt} \Gamma_{\lambda+\mu}(y^\varepsilon_{\lambda\mu})
  + \ip[\big]{Ay^\varepsilon_{\lambda\mu}}%
       {\gamma_{\lambda+\mu}(y^\varepsilon_{\lambda\mu})}
  + \ip[\big]{g_{\lambda\mu}^\varepsilon}%
    {\gamma_{\lambda+\mu}(y^\varepsilon_{\lambda\mu})} = 0.
  \]
  Since \(A\) is the generator of a subMarkovian semigroup of
  contractions and \(\gamma_{\lambda+\mu}\) is an increasing function
  with \(\gamma_{\lambda+\mu}(0)=0\), it follows by a lemma of
  Br\'ezis and Strauss (see \cite[Lemma~2]{BreStr}) that the second
  term on the left-hand side is positive, hence
  \[
    \Gamma_{\lambda+\mu}(y^\varepsilon_{\lambda\mu})
    + \int_0^\cdot \ip[\big]{g_{\lambda\mu}^\varepsilon}%
    {\gamma_{\lambda+\mu}(y^\varepsilon_{\lambda\mu})} \leq 0,
  \]
  thus also, taking the limit as \(\varepsilon \to 0\),
  \[
    \Gamma_{\lambda+\mu}(u_\lambda - u_\mu)
    + \int_0^\cdot \ip[\big]{f_\lambda(u_\lambda) - f_\mu(u_\mu)}%
    {\gamma_{\lambda+\mu}(u_\lambda - u_\mu)} \leq 0.
  \]
  Moreover, writing
  \begin{align*}
    \gamma_{\lambda+\mu} \bigl( u_\lambda - u_\mu \bigr)
    &= \gamma_{\lambda+\mu} \bigl( R_\lambda u_\lambda - R_\mu u_\mu \bigr)\\
    &\quad + \gamma_{\lambda+\mu} \bigl( u_\lambda - u_\mu \bigr)
      - \gamma_{\lambda+\mu} \bigl( R_\lambda u_\lambda - R_\mu u_\mu \bigr),
  \end{align*}
  one has
  \begin{align*}
    &\bigl( f_\lambda(u_\lambda) - f_\mu(u_\mu) \bigr)
      \gamma_{\lambda+\mu}(u_\lambda-u_\mu)\\
    &\qquad \in \bigl( f(R_\lambda u_\lambda) - f(R_\mu u_\mu) \bigr)
      \gamma_{\lambda+\mu} \bigl( R_\lambda u_\lambda - R_\mu u_\mu \bigr)\\
    &\qquad \quad  + \bigl( f_\lambda(u_\lambda) - f_\mu(u_\mu) \bigr)
      \Bigl( \gamma_{\lambda+\mu} \bigl( u_\lambda - u_\mu \bigr)
      - \gamma_{\lambda+\mu} \bigl( R_\lambda u_\lambda - R_\mu u_\mu \bigr)
      \Bigr)\\
    &\qquad \geq \bigl( f_\lambda(u_\lambda) - f_\mu(u_\mu) \bigr)
      \Bigl( \gamma_{\lambda+\mu} \bigl( u_\lambda - u_\mu \bigr)
      - \gamma_{\lambda+\mu} \bigl( R_\lambda u_\lambda - R_\mu u_\mu \bigr)
      \Bigr),
  \end{align*}
  where, by definition of \((\gamma_\varepsilon)\),
  \begin{align*}
    &\abs[\Big]{\gamma_{\lambda+\mu} \bigl( u_\lambda - u_\mu \bigr)
      - \gamma_{\lambda+\mu} \bigl( R_\lambda u_\lambda - R_\mu u_\mu \bigr)}\\
    &\hspace{3em}  
      \leq \biggl( \frac{1}{\sqrt{\lambda+\mu}}
      \abs[\big]{u_\lambda - R_\lambda u_\lambda - (u_\mu - R_\mu u_\mu)}
      \biggr) \wedge 2,
  \end{align*}
  and
  \begin{align*}
    \abs[\big]{u_\lambda - R_\lambda u_\lambda - (u_\mu - R_\mu u_\mu)}
    &= \abs[\big]{\lambda f_\lambda(u_\lambda) - \mu f_\mu(u_\mu)}\\
    &\lesssim (\lambda + \mu) \bigl( \abs{f_\lambda(u_\lambda)}
      + \abs{f_\mu(u_\mu)} \bigr),
  \end{align*}
  hence
  \begin{align*}
    &\abs[\Big]{\gamma_{\lambda+\mu} \bigl( u_\lambda - u_\mu \bigr)
      - \gamma_{\lambda+\mu} \bigl( R_\lambda u_\lambda - R_\mu u_\mu \bigr)}\\
    &\hspace{3em}  
      \lesssim \Bigl( \sqrt{\lambda+\mu} \, \bigl( \abs{f_\lambda(u_\lambda)}
      + \abs{f_\mu(u_\mu)} \bigr) \Bigr) \wedge 2.
  \end{align*}
  Setting \(T_2\colon x \mapsto \abs{x} \wedge 2\), this implies
  \[
    \Gamma_{\lambda+\mu}(u_\lambda-u_\mu) \lesssim
    \int_0^\cdot\!\!\int_G \bigl( \abs{f_\lambda(u_\lambda)}
    + \abs{f_\mu(u_\mu)} \bigr) T_2\Bigl( \sqrt{\lambda+\mu} \,
    \bigl( \abs{f_\lambda(u_\lambda)} + \abs{f_\mu(u_\mu)} \bigr)
    \Bigr),
  \]
  thus also
  \begin{equation}
    \label{eq:T2}
    \norm[\big]{\Gamma_{\lambda+\mu}(u_\lambda-u_\mu)}_{C([0,T])} \lesssim
    \int_0^T\!\!\int_G \bigl( \abs{f_\lambda(u_\lambda)}
    + \abs{f_\mu(u_\mu)} \bigr) T_2\Bigl( \sqrt{\lambda+\mu} \,
    \bigl( \abs{f_\lambda(u_\lambda)} + \abs{f_\mu(u_\mu)} \bigr)
    \Bigr).
  \end{equation}
  We are going to show that the right-hand side goes to zero as
  \(\lambda,\mu \to 0\). In fact, by Markov's inequality,
  \begin{align*}
    \operatorname{Leb}\Bigl( \sqrt{\lambda+\mu} \, \abs{f_\lambda(u_\lambda)}
    \geq \varepsilon \Bigr)
    &= \operatorname{Leb}\biggl( \abs{f_\lambda(u_\lambda)}
      \geq \frac{\varepsilon}{\sqrt{\lambda+\mu}} \biggr)\\
    &\leq \frac{\sqrt{\lambda+\mu}}{\varepsilon} \,
      \norm[\big]{f_\lambda(u_\lambda)}_{L^1([0,T] \times G)}
  \end{align*}
  for every \(\varepsilon \in \erre_+\), hence, by the boundedness of
  \((f_\lambda(u_\lambda))\) in \(L^1([0,T] \times G)\),
  \(\sqrt{\lambda+\mu} (\abs{f_\lambda(u_\lambda)} +
  \abs{f_\mu(u_\mu)})\) tends to zero in measure as \(\lambda\) and
  \(\mu\) tend to zero. The continuous mapping theorem then implies
  \[
    T_2\Bigl( \sqrt{\lambda+\mu} \, \bigl( \abs{f_\lambda(u_\lambda)}
    + \abs{f_\mu(u_\mu)} \bigr) \Bigr) \longrightarrow 0
  \]
  in measure as \(\lambda\) and \(\mu\) tend to zero. As
  \((f_\lambda(u_\lambda))\) is equiintegrable by hypothesis, Lemma
  \ref{lm:EIconv} implies, in view of \eqref{eq:T2}, that
  \(\Gamma_{\lambda+\mu}(u_\lambda-u_\mu)\) tends to zero in
  \(C([0,T])\) as \(\lambda,\mu \to 0\).
  Since
  \[
    \norm[\big]{u_\lambda-u_\mu}_{C([0,T];L^1)} \leq
    \norm[\big]{{\norm{u_\lambda-u_\mu}}_{L^1} -
      \Gamma_{\lambda+\mu}(u_\lambda-u_\mu)}_{C([0,T])} +
    \norm[\big]{\Gamma_{\lambda+\mu}(u_\lambda-u_\mu)}_{C([0,T])},
  \]
  where the second term on the right-hand side has just been shown to
  tend to zero as \(\lambda,\mu \to 0\), the proof is complete if one
  proves that
  \[
    \lim_{\varepsilon \to 0} \Gamma_\varepsilon(\phi) = {\norm{\phi}}_{L^1}
  \]
  in \(C([0,T])\) uniformly with respect to \(\phi\) on bounded sets
  of \(L^1([0,T] \times G)\). The definition of
  \(\Gamma^0_\varepsilon\) implies
  \begin{align*}
    \norm[\big]{\Gamma_\varepsilon(\phi) - {\norm{\phi}}_{L^1}}_{C([0,T])}
    &\leq \sup_{t \in [0,T]}
      \int_G \abs[\big]{\Gamma^0_\varepsilon(\phi(t,x)) - \abs{\phi(t,x)}}\,dx\\
    &= \sup_{t \in [0,T]} \int_{A_\varepsilon(t)}
      \abs[\big]{\Gamma^0_\varepsilon(\phi(t,x)) - \abs{\phi(t,x)}}\,dx,
  \end{align*}
  where
  \( A_\varepsilon(t) := \bigl\{ x \in G:\, \abs{\phi(t,x)} \leq
  \sqrt{\varepsilon}/2 \bigr\}\). Recalling that the distance in
  \(L^\infty(\erre)\) between \(\Gamma^0_\varepsilon\) and
  \(\abs{\cdot}\) is bounded by \(\sqrt{\varepsilon}/4\), we get
  \[
    \norm[\big]{\Gamma_\varepsilon(\phi) - {\norm{\phi}}_{L^1}}_{C([0,T])}
    \leq \sup_{t \in [0,T]} \frac{\sqrt{\varepsilon}}{4} \abs{A_\varepsilon(t)}
    \leq \frac{\sqrt{\varepsilon}}{4} \abs{G},
  \]
  that establishes the needed uniform convergence and concludes the
  proof.
\end{proof}

\subsection{Existence and uniqueness}
We establish existence and uniqueness of \((q,r)\)-mild solutions, while
generalized solutions will be discussed separately.
Let us begin with uniqueness of \(L^1\)-valued mild solutions, that
obviously implies also uniqueness of \((q,r)\)-mild solutions for
every \(q,r \in [1,+\infty\mathclose[\).
\begin{prop}
  \label{prop:contr}
  Let \(q,r \in [1,+\infty\mathclose[\), \(r \leq q\), and
  \(p \in \erre_+\). Assume that \(u^1_0,u^2_0 \in L^p(\cF_0;L^q)\),
  \(u^1,u^2 \in L^p(\Omega;C([0,T];L^q))\) and
  \(g^1,g^2 \in L^0(\Omega;L^1(0,T;L^r)\) are adapted processes such
  that \(g^1 \in f(u^1)\) and \(g^2 \in f(u^2)\) a.e. on
  \(\Omega \times [0,T] \times G\), and
  \[
    u^1 + S \ast g^1 = Su^1_0 + S \diamond B, \qquad
    u^2 + S \ast g^2 = Su^2_0 + S \diamond B.
  \]
  Then
  \[
    \norm[\big]{u^1 - u^2}_{L^p(\Omega;C([0,T];L^r))}
    \leq \norm[\big]{u^1_0-u^2_0}_{L^p(\Omega;L^r)}.
  \]  
\end{prop}
\begin{proof}
  The process \(u^1-u^2\) satisfies
  \(u^1-u^2 + S \ast (g^1-g^2) = S(u^1_0-u^2_0)\). Denoting the norm
  of \(L^r\) by \(\norm{\cdot}\), Proposition \ref{prop:CL} yields
  \[
    \norm[\big]{u^1-u^2} \leq \norm[\big]{u^1_0 - u^2_0} -
    \int_0^\cdot \bigl[ u^1 - u^2, g^1 - g^2 \bigr],
  \]
  where, by accretivity of \(f\) in \(L^r\), the integral in the above
  inequality is positive, hence
  \(\norm{u^1-u^2} \leq \norm{u^1_0 - u^2_0}\), from which the claim
  follows immediately.
\end{proof}
\begin{rmk}
  If \(r \in \mathopen]1,+\infty\mathclose[\), it suffices to use the
  more elementary estimate of Proposition \ref{prop:stm} in place of
  Proposition \ref{prop:CL}.
\end{rmk}
Taking \(p=0\) and \(q=r=1\) yields uniqueness of solutions.
\begin{coroll}
  If equation \eqref{eq:0} admits an \(L^1\)-valued mild solution,
  then it is unique.
\end{coroll}

Let us now consider existence of solutions.
For any \(q,r \in [1,+\infty\mathclose[\), let
\[
q^\ast :=
\begin{cases}
  rd \vee (2d+q-2), & \text{if } q \in [2,+\infty\mathclose[,\\
  qd, & \text{if } q \in \mathopen]1,2\mathclose[.\\
\end{cases}
\]
\begin{thm}
  \label{thm:exist}
  Let \(q,r \in \mathopen]1,+\infty\mathclose[\). If
  \(u_0 \in L^0(\cF_0;L^{q^\ast})\) and
  \[
    S \diamond B \in L^0(\Omega;C([0,T];L^{q^\ast})) \cap
    L^0(\Omega;L^d(0,T;L^{dq^\ast})),
  \]
  then there exists a unique \((q,r)\)-mild solution to
  \eqref{eq:0}.
\end{thm}
\begin{proof}
  Proposition \ref{prop:ape1} implies that \((u_\lambda)\) is bounded
  in \(C([0,T];L^{q^\ast})\) almost surely.  Therefore, thanks to
  Lemmas \ref{lm:C2+} and \ref{lm:C1+}, \(u_\lambda\) is a Cauchy
  sequence in \(C([0,T];L^q)\) almost surely, hence there exists a
  continuous adapted process \(u\) such that \(u_\lambda\) converges
  to \(u\) in \(C([0,T];L^q)\) almost surely.
  Since \(\abs{f_\lambda(u_\lambda)} \lesssim 1 + \abs{u_\lambda}^d\),
  one has
  \[
    \norm[\big]{f_\lambda(u_\lambda)}_{L^r} \lesssim 1 +
    \norm[\big]{u_\lambda}^d_{L^{rd}}
    \lesssim 1 + \norm[\big]{u_\lambda}^d_{L^{q^\ast}},
  \]
  thus also
  \[
    \norm[\big]{f_\lambda(u_\lambda(t))}_{C([0,T];L^r)}
    \lesssim 1 + \norm[\big]{u_\lambda}^d_{C([0,T];L^{q^\ast})},
  \]
  with an implicit constant depending on \(\abs{G}\).  Hence there
  exists an event \(\Omega_0\) with \(\P(\Omega_0)=1\) such that
  \(g_\lambda(\omega) := f_\lambda(u_\lambda(\omega))\) is bounded in
  \(L^\infty(0,T;L^r)\), thus also in \(L^r(0,T;L^r)\), for every
  \(\omega \in \Omega_0\). Therefore, for each
  \(\omega \in \Omega_0\), there exists a subsequence \(\lambda'\) of
  \(\lambda\) and \(g \in L^\infty(0,T;L^r)\), both depending on
  \(\omega\), such that \(g_{\lambda'}\) converges to \(g\) in the
  weak* topology of \(L^\infty(0,T;L^r)\) and in the weak topology of
  \(L^r(0,T;L^r)\). As is easy to see, \(\phi \mapsto S \ast \phi\) is
  sequentially weak* continuous on \(L^\infty(0,T;L^r)\), which yields
  \[
    u + S \ast g = Su_0 + S \diamond B
  \]
  for every \(\omega \in \Omega_0\), as an identity in
  \(C([0,T];L^r)\). We are going to show that \(g\) is a predictable
  process by a uniqueness argument: let \(g_1\) and \(g_2\) be two
  different adherent points of \((g_\lambda)\). Then
  \[
    u + S \ast g_i = Su_0 + S \diamond B, \qquad i=1,2,
  \]
  hence \(S\ast (g_1-g_2)=0\), which implies
  \(S(t-s)(g_1(s)-g_2(s)) = 0\) for a.a. \(s \in [0,t]\) for a.a.
  \(t \in [0,T]\). Since the kernel of a strongly continuous semigroup
  is trivial,\footnote{If \(S(h)x=0\) for \(h\) in a right
    neighborhood of zero, \(x = \lim_{h \to 0} S(h)x = 0\).}  it
  follows that \(g_1 = g_2\) in \(L^\infty(0,T;L^r)\). Recalling once
  again that \(L^\infty(0,T;L^r)\) is compact in the weak* topology,
  it follows that the whole sequence \(g_\lambda\) converges to \(g\)
  (cf, e.g., \cite[TG~I.60, Corollaire]{Bbk}). In particular,
  \(g_\lambda\) converges to \(g\) weakly in \(L^r(0,T;L^r)\), hence,
  by Mazur's lemma, there exists a sequence \((h_n)\) of convex
  combinations of \(g_\lambda\) that converges to \(g\) strongly in
  \(L^r(0,T;L^r)\). Then, for every \(\omega \in \Omega_0\), a
  subsequence of \((h_n)\) converges to \(g\) almost everywhere on
  \([0,T]\), hence it converges pointwise on \([0,T]\) to a function
  in the same Lebesgue equivalence class of \(g\), still denoted by
  \(g\). Therefore, as each \(h_n\) is adapted and continuous, hence
  predictable, \(g\) is itself predictable. It remains to show that
  \(g \in f(u)\) a.e. in \([0,T] \times G\). To this purpose, setting
  \(s:=q \wedge r\), it suffices to recall that \(f\) is
  \(m\)-accretive in \(L^s([0,T] \times G)\), hence it is
  strongly-weakly closed, and to note that
  \(g_\lambda = f_\lambda(u_\lambda) \in f(R_\lambda u_\lambda)\),
  with \(R_\lambda u_\lambda\) converging to \(u\) in
  \(L^q([0,T] \times G)\), thus also in \(L^s([0,T] \times G)\). As
  \(g_\lambda\) converges to \(g\) weakly in \(L^r([0,T] \times G)\),
  hence also weakly in \(L^s([0,T] \times G)\), it follows that
  \(g \in f(u)\) a.e. in \([0,T] \times G\).
\end{proof}

We are going to prove existence and uniqueness of \(L^1\)-valued mild
solutions in a conditional sense, for reasons discussed after the proof.
\begin{thm}
  Assume that \(S\) is subMarkovian, \((0,0) \in f\),
  \(u_0 \in L^0(\cF_0;L^1)\), and
  \[
    S \diamond B \in L^0(\Omega;C([0,T];L^1)).
  \]
  If \((f_\lambda(u_\lambda))\) is equiintegrable on
  \([0,T] \times G\) and \((f_\lambda(u_\lambda)u_\lambda)\) is
  bounded in \(L^1([0,T] \times G)\), then there exist a unique mild
  solution to \eqref{eq:0}.
\end{thm}
\begin{proof}
  Proposition \ref{prop:C1} implies that \((u_\lambda)\) is a Cauchy
  sequence in \(C([0,T];L^1)\), hence it admits a unique limit
  \(u\). The equiintegrability assumption also implies, thanks to the
  Dunford-Pettis theorem, that, for every \(\omega \in \Omega\)
  outside a negligible set, \((g_\lambda) := (f_\lambda(u_\lambda))\)
  is weakly compact in \(L^1([0,T] \times G)\), hence it admits an
  accumulation point \(g(\omega) \in L^1([0,T] \times G)\).  By
  arguments entirely analogous to those used in the proof of Theorem
  \ref{thm:exist}, it turns out that \(g\) is unique, it is a
  predictable process, and satisfies the identity
  \[
    u + S \ast g = Su_0 + S \diamond B.
  \]
  It remains to show that \(g \in f(u)\) a.e. on \([0,T] \times G\).
  Let us write
  \[
    f_\lambda(u_\lambda)u_\lambda = f_\lambda(u_\lambda)R_\lambda u_\lambda
    + f_\lambda(u_\lambda)(u_\lambda - R_\lambda u_\lambda).
  \]
  Recalling that \(0 \in f(0)\) and \(R_\lambda\) is a positive
  contraction of \(\erre\), it is easy to see that
  \[
    \abs[\big]{f_\lambda(u_\lambda)(u_\lambda - R_\lambda u_\lambda)}
    \leq 2f_\lambda(u_\lambda)u_\lambda,
  \]
  hence the second term on the right-hand side of the previous
  identity is bounded in \(L^1([0,T] \times G)\) and converges to zero
  as \(\lambda \to 0\) by Vitali's theorem.  In particular,
  \((f_\lambda(u_\lambda)R_\lambda u_\lambda) \in (f(R_\lambda
  u_\lambda)R_\lambda u_\lambda)\) is bounded in
  \(L^1([0,T] \times G)\) and converges to \(g\) weakly. Since,
  possibly on a subsequence, \(u_\lambda \to u\) a.e. on
  \([0,T] \times G\), a lemma by Br\'ezis (see
  \cite[Theorem~18]{Bre-mm}) implies that \(g \in f(u)\) a.e. on
  \([0,T] \times G\), thus also a.e. on
  \(\Omega \times [0,T] \times G\).
\end{proof}
\begin{rmk}
  The hypotheses of the theorem have been formulated in indirect terms
  to emphasize what is really needed for the argument to work. In
  particular, the power-like growth of \(f\) does not play any
  role. This is interesting because it can be shown (cf., e.g.,
  \cite{Barbu:lincei}) that if the stochastic convolution is
  continuous in space and time, then the sufficient conditions of the
  theorem are fulfilled assuming only, \emph{grosso modo}, that the
  range of (the maximal monotone graph associated to) \(f\) is
  \(\erre\). On the other hand, if the stochastic convolution is
  bounded on \([0,T] \times G\), it immediately follows from Theorem
  \ref{thm:exist} that \eqref{eq:0} admits a unique \((q,r)\)-mild
  solution for every \(q \in \mathopen]1,+\infty\mathclose[\) and
  every \(d \in \erre_+\). It is not clear, however, whether this
  suffices to circumvent the power-like growth condition on \(f\).
  Furthermore, it is clear that, under the usual growth assumption on
  \(f\), if \((u_\lambda)\) is bounded in \(L^{d+1}([0,T] \times G)\),
  then the hypotheses of the theorem are met, hence there exists a
  unique \(L^1\)-valued mild solution to \eqref{eq:0}. This is,
  however, not optimal, at least if one assumes a bit more
  integrability on the stochastic convolution: setting
  \(q:=(d+1)/d>1\), one has that \((u_\lambda)\) is bounded in
  \(L^{qd}\), which implies (assuming, for simplicity, \(d \geq 1\))
  that \(u_\lambda\) is a Cauchy sequence in \(C([0,T];L^q)\) and that
  \((f_\lambda(u_\lambda))\) is bounded in \(L^q([0,T] \times G)\),
  hence, by the reasoning of Theorem \ref{thm:exist}, that
  \eqref{eq:0} admits a unique \(L^q\)-valued strict mild solution. A
  closer investigation of these issues will hopefully appear
  elsewhere.
\end{rmk}

\subsection{Further properties of mild solutions}
Proposition \ref{prop:contr} immediately implies that the solution map
\(u_0 \mapsto u\) for \((q,r)\)-mild solutions is a contraction from
\(L^p(\Omega;L^r)\) to \(L^p(\Omega;C([0,T];L^r))\) for every
\(p \in \erre_+\) and every \(r,q \in [1,+\infty\mathclose[\).
One can deduce further estimates on the solution map using the
construction of solutions via the convergence of solutions to
regularized equations. In this case, however, assumptions on the
stochastic convolution have to be made.
\begin{prop}
  \label{prop:soma}
  Let \(q,r \in \mathopen]1,+\infty\mathclose[\) and assume that
  \[
    S \diamond B \in L^0(\Omega;C([0,T];L^{q^\ast})) \cap
    L^0(\Omega;L^d(0,T;L^{dq^\ast})).
  \]
  The solution map
  \begin{align*}
    L^0(\cF_0;L^{q^\ast}) &\longrightarrow L^0(\Omega;C([0,T];L^q))\\
    u_0 &\longmapsto u
  \end{align*}
  for \((q,r)\)-mild solutions is a contraction from
  \(L^p(\cF_0;L^q)\) to \(L^p(\Omega;C([0,T];L^q))\) for every
  \(p \in \erre_+\).
\end{prop}
\begin{proof}
  Let \(u^1_0, u^2_0 \in L^0(\cF_0;L^{q^\ast})\), and \(u^1_\lambda\),
  \(u^2_\lambda\) be the \((q,r)\)-mild solutions to the regularized
  equation \eqref{eq:reg} with initial conditions \(u^1_0\),
  \(u^2_0\), respectively. Then one has
  \[
    (u^1_\lambda - u^2_\lambda) + S \ast \bigl( f_\lambda(u^1_\lambda)
    - f_\lambda(u^2_\lambda) \bigr) = S(u^1_0 - u^2_0),
  \]
  hence, by Proposition \ref{prop:stm},
  \[
    \norm[\big]{u^1_\lambda - u^2_\lambda}_{C([0,T];L^q))}
    \leq \norm[\big]{u^1_0 - u^2_0}_{L^q}.
  \]
  Lemmas \ref{lm:C2+} and \ref{lm:C1+} imply that
  \(u^i_\lambda \to u^i\), \(i=1,2\), strongly in \(C([0,T];L^q)\),
  from which the claim follows immediately.
\end{proof}

The proposition implies an estimate of \(u\) in terms of the initial
datum \(u_0\) as follows: denoting the solution with initial datum
equal to zero by \(u^0\), one has
\begin{align*}
  {\norm{u}}_{C([0,T];L^q)}
  &\leq {\norm{u-u^0}}_{C([0,T];L^q)}
    + {\norm{u^0}}_{C([0,T];L^q)}\\
  &\leq {\norm{u_0}}_{L^q}
    + {\norm{u^0}}_{C([0,T];L^q)},
\end{align*}
where, by Theorem \ref{thm:exist}, the second term on the right-hand
side is finite almost surely. It should be noted that one cannot infer
from this inequality estimates of the type
\[
  {\norm{u}}_{L^p(\Omega;C([0,T];L^q))} \lesssim 1 +
  {\norm{u_0}}_{L^p(\Omega;L^q)},
\]
unless conditions are provided implying that \(u^0\) belongs to
\(L^p(\Omega;C([0,T];L^q))\).

It is natural to ask whether the integrability in space of the initial
datum ``propagates'' to the solution. This indeed the case: the mild
solution inherits the integrability in space of the initial datum,
uniformly with respect to time, as we show next.
\begin{prop}
  Assume that \(q^\ast \geq q\), that the hypotheses of Theorem
  \ref{thm:exist} are satisfied, and define the random variable
  \[
    \xi := \norm[\big]{S \diamond B}_{C([0,T];L^{q^\ast})}
    + \norm[\big]{S \diamond B}^d_{L^d(0,T;L^{dq^\ast})}.
  \]
  Then
  \[
    {\norm{u}}_{L^\infty(0,T;L^{q^\ast})} \lesssim 1 + \xi +
    {\norm{u_0}}_{L^{q^\ast}}
  \]
  almost surely.
\end{prop}
\begin{proof}
  As in the proof of Proposition \ref{prop:ape1}, one has, in the
  almost sure sense,
  \[
    {\norm{u_\lambda}}_{C([0,T];L^{q^\ast})} \lesssim 1 +
    {\norm{u_0}}_{L^{q^\ast}} + \xi,
  \]
  that is, as \(\xi\) is finite almost surely by assumption,
  \(u_\lambda\) is bounded in \(L^\infty(0,T;L^{q^\ast})\) almost
  surely.  Therefore, for each \(\omega \in \Omega\) outside a
  negligible set, \(u_\lambda(\omega)\) is weakly* compact in
  \(L^\infty(0,T;L^{q^\ast})\), i.e. there exists
  \(v(\omega) \in L^\infty(0,T;L^{q^\ast})\) and a subsequence
  \(\lambda'(\omega)\) of \(\lambda\) such that \(u_{\lambda'} \to v\)
  in the weak* topology of \(L^\infty(0,T;L^{q^\ast})\), hence, a
  fortiori, in the weak* topology of \(L^\infty(0,T;L^q)\). Since
  \((u_\lambda)\) converges to \(u\) strongly in \(C([0,T];L^q)\), one
  infers that \(v=u\) as elements of \(L^\infty(0,T;L^q)\), in
  particular \(v=u\) a.e. on \(\Omega \times [0,T] \times G\), hence
  \(u \in L^\infty(0,T;L^{q^\ast})\) almost surely. More precisely, by
  the weak* lower semicontinuity of the norm,
  \[
    {\norm{u}}_{L^\infty(0,T;L^{q^\ast})} \leq \liminf_{\lambda
      \to 0}{} {\norm{u_\lambda}}_{C([0,T];L^{q^\ast})} \lesssim 1 + \xi +
    {\norm{u_0}}_{L^{q^\ast}}.
    \qedhere
  \]
\end{proof}
\begin{coroll}
  Let \(q,r \in \mathopen]1,+\infty\mathclose[\). If
  \(u_0 \in L^p(\cF_0;L^{q^\ast})\) and
  \[
    S \diamond B \in L^p(\Omega;C([0,T];L^{q^\ast})) \cap
    L^p(\Omega;L^d(0,T;L^{dq^\ast})),
  \]
  then there exists a unique \((q,r)\)-mild solution
  \(u \in L^p(\Omega;C([0,T];L^q))\) to \eqref{eq:0}. Moreover,
  \[
    \norm[\big]{u}_{L(\Omega;L^\infty(0,T;L^{q^\ast}))} \lesssim 1 +
    \norm[\big]{u_0}_{L^p(\Omega;L^{q^\ast})}.
  \]
\end{coroll}

\subsection{Generalized solutions}
The existence and uniqueness of generalized solutions to \eqref{eq:0}
can be established as a consequence of Proposition \ref{prop:soma}.
\begin{thm}
  Let \(q \in \mathopen]1,+\infty\mathclose[\), \(r=q\). Assume that
  \(u_0 \in L^0(\cF_0;L^q)\) and
  \[
    S \diamond B \in L^0(\Omega;C([0,T];L^{q^\ast})) \cap
    L^0(\Omega;L^d(0,T;L^{dq^\ast})).
  \]
  There exists a unique generalized solution
  \(u \in L^0(\Omega;C([0,T];L^q))\) to \eqref{eq:0}. Moreover, the
  solution map \(u_0 \mapsto u\) is a contraction from
  \(L^p(\Omega;L^q)\) to \(L^p(\Omega;C([0,T];L^q))\) for every
  \(p \in \erre_+\).
\end{thm}
\begin{proof}
  Let us reason pathwise, i.e. fixing an arbitrary element of
  \(\Omega\) outside a negligible set.  The solution map
  \(u_0 \mapsto u\colon L^{q^\ast} \to C([0,T];L^q)\) is well defined
  by Theorem \ref{thm:exist}, and, by Proposition \ref{prop:soma}, it
  is a contraction if the domain is endowed with the \(L^q\) norm. In
  particular, the solution map is uniformly continuous on a dense
  subset of \(L^q\). As the space \(C([0,T];L^q)\) is separated and
  complete, the map \(u_0 \mapsto u\) admits a unique extension to a
  uniformly continuous map \(L^q \to C([0,T];L^q)\) (cf, e.g.,
  \cite[TG II.20, Th\'eor\`eme 2]{Bbk}). By continuity, this unique
  extension is still a contraction.
\end{proof}

Note that the integrability hypothesis on the initial datum is the
natural one, which is instead not the case in Theorem
\ref{thm:exist}. In contrast to mild solutions, however, if \(u\) is a
generalized solution, the deterministic convolution term in
\eqref{eq:0} may not be defined, as \(f(u)\), in general, just belongs
to \(L^0(\Omega;L^\infty(0,T;L^{q/d}))\), and \(L^{q/d}\) is not a
Banach space if \(q<d\).

\bibliographystyle{amsplain}
\bibliography{ref}

\def\polhk#1{\setbox0=\hbox{#1}{\ooalign{\hidewidth
  \lower1.5ex\hbox{`}\hidewidth\crcr\unhbox0}}}
\providecommand{\bysame}{\leavevmode\hbox to3em{\hrulefill}\thinspace}
\providecommand{\MR}{\relax\ifhmode\unskip\space\fi MR }
\providecommand{\MRhref}[2]{%
  \href{http://www.ams.org/mathscinet-getitem?mr=#1}{#2}
}
\providecommand{\href}[2]{#2}
\begin{thebibliography}{10}

\bibitem{Asp:dm}
E.~Asplund, \emph{Positivity of duality mappings}, Bull. Amer. Math. Soc.
  \textbf{73} (1967), 200--203. \MR{206663}

\bibitem{barbu}
V.~Barbu, \emph{Analysis and control of nonlinear infinite-dimensional
  systems}, Academic Press Inc., Boston, MA, 1993. \MR{MR1195128 (93j:49002)}

\bibitem{Barbu:lincei}
\bysame, \emph{Existence for semilinear parabolic stochastic equations}, Atti
  Accad. Naz. Lincei Cl. Sci. Fis. Mat. Natur. Rend. Lincei (9) Mat. Appl.
  \textbf{21} (2010), no.~4, 397--403. \MR{2746091 (2012d:35424)}

\bibitem{Bbk}
N.~Bourbaki, \emph{{\'E}l\'ements de math\'ematique}, Springer, 2006--present.

\bibitem{Bre-mm}
H.~Br{\'e}zis, \emph{Monotonicity methods in {H}ilbert spaces and some
  applications to nonlinear partial differential equations}, Contributions to
  nonlinear functional analysis (Proc. Sympos., Math. Res. Center, Univ.
  Wisconsin, Madison, Wis., 1971), Academic Press, New York, 1971,
  pp.~101--156. \MR{MR0394323 (52 \#15126)}

\bibitem{BreStr}
H.~Br{\'e}zis and W.~A. Strauss, \emph{Semi-linear second-order elliptic
  equations in {$L^{1}$}}, J. Math. Soc. Japan \textbf{25} (1973), 565--590.
  \MR{0336050 (49 \#826)}

\bibitem{cerrai03}
S.~Cerrai, \emph{Stochastic reaction-diffusion systems with multiplicative
  noise and non-{L}ipschitz reaction term}, Probab. Theory Related Fields
  \textbf{125} (2003), no.~2, 271--304. \MR{1961346 (2004a:60117)}

\bibitem{DP:K}
G.~Da~Prato, \emph{Kolmogorov equations for stochastic {PDE}s}, Birkh\"auser
  Verlag, Basel, 2004. \MR{2111320 (2005m:60002)}

\bibitem{DPZ}
G.~Da~Prato and J.~Zabczyk, \emph{Stochastic equations in infinite dimensions},
  second ed., Cambridge University Press, Cambridge, 2014. \MR{3236753}

\bibitem{KvN2}
M.~Kunze and J.~van Neerven, \emph{Continuous dependence on the coefficients
  and global existence for stochastic reaction diffusion equations}, J.
  Differential Equations \textbf{253} (2012), no.~3, 1036--1068. \MR{2922662}

\bibitem{cm:SIMA18}
C.~Marinelli, \emph{On well-posedness of semilinear stochastic evolution
  equations on {$L_p$} spaces}, SIAM J. Math. Anal. \textbf{50} (2018), no.~2,
  2111--2143. \MR{3784905}

\bibitem{cm:AP18}
C.~Marinelli and L.~Scarpa, \emph{A variational approach to dissipative {SPDE}s
  with singular drift}, Ann. Probab. \textbf{46} (2018), no.~3, 1455--1497.
  \MR{3785593}

\bibitem{vNVW}
J.~van Neerven, M.~C. Veraar, and L.~Weis, \emph{Stochastic evolution equations
  in {UMD} {B}anach spaces}, J. Funct. Anal. \textbf{255} (2008), no.~4,
  940--993. \MR{2433958 (2009h:35465)}

\bibitem{Walter:ODE}
W.~Walter, \emph{Ordinary differential equations}, Springer Verlag, New York,
  1998. \MR{1629775 (2001k:34002)}

\end{thebibliography}

\end{document}